\documentclass[11pt]{article}
\usepackage{amsmath,amsthm, amssymb, amsfonts,accents,nccmath}
\usepackage{enumitem}
\usepackage{array}
\usepackage[toc,page]{appendix}
\usepackage[english]{babel}
\usepackage{dsfont}
\usepackage{booktabs}
\usepackage[linesnumbered,commentsnumbered,ruled,vlined]{algorithm2e}
\usepackage{tikz}
\usepackage{tikz-network}
\usepackage{caption}
\usepackage{subcaption}
\usepackage[utf8]{inputenc}
\usepackage[english]{babel}

\usepackage{caption}
\usepackage{subcaption}

\usepackage{cite}
 
\usepackage{hyperref}
\hypersetup{
    colorlinks=true,
    linkcolor=blue,
    filecolor=darkmagenta,      
    urlcolor=cyan,
    citecolor=red
} 
\makeatletter
\def\BState{\State\hskip-\ALG@thistlm}
\makeatother
\numberwithin{equation}{section}

\newtheorem{theorem}{Theorem}[section]
\newtheorem{cor}[theorem]{Corollary}
\newtheorem{lem}[theorem]{Lemma}
\newtheorem{prop}[theorem]{Proposition}

\newtheorem{rem}[theorem]{Remark}

\newcommand{\ds}{\displaystyle}

\newcommand{\cof}{\textup{cof }}

\usepackage{mathtools}
\DeclarePairedDelimiter\ceil{\lceil}{\rceil}
\DeclarePairedDelimiter\floor{\lfloor}{\rfloor}

\title{On the Squared Distance Matrix of a Starlike Block Graph}
\author{Joyentanuj Das\footnote{Department of Mathematics, College of Engineering and Technology, SRM Institute of Science and Technology, Kattankulathur, Chennai, 603203, India. \indent  Email: joyentanuj@gmail.com,  joyentad@srmist.edu.in}  \quad and \quad Sumit Mohanty\footnote{Indian Institute of Management Ranchi,  Prabandhan Nagar, Vill-Mudma, Nayasarai Road, Ranchi, Jharkhand-835303, India. \indent   Email:  sumitmath@gmail.com, sumit.mohanty@iimranchi.ac.in}}
\date{}

\begin{document}

\maketitle

\begin{abstract}
Let $D(G)$ be the distance matrix of a simple connected graph $G$. The  Hadamard product  $D(G)~\circ~ D(G)$ is called the squared distance matrix of $G$, and is denoted by $\Delta(G)$. A simple connected graph is called a starlike block graph if it has a central cut vertex,  and each of its blocks is a complete graph. Let $ \mathcal{S}(n_1, n_2, \ldots, n_b)$ be the starlike block graph with blocks $K_{n_1+1}, K_{n_2+1}, \ldots, K_{n_b+1} $ on $n=1 + \sum_{i=1}^b n_i$ vertices. In this article, we compute the determinant of $\Delta( \mathcal{S}(n_1, n_2, \ldots, n_b))$ and find its inverse as a rank-one perturbation of a positive semidefinite Laplacian-like matrix $\mathcal{L}$ with rank $n-1$. We also investigate the inertia of $\Delta( \mathcal{S}(n_1, n_2, \ldots, n_b))$. Furthermore, for a fixed value of $ n $ and $ b $, we determine the extremal graphs that uniquely attain the maximum and minimum spectral radius of the squared distance matrix for starlike block graphs on $ n $ vertices and $ b $ blocks.
\end{abstract}


\section{Introduction  and Motivation}
Let $G$ be a simple connected graph with vertex set $\{v_1,v_2,\ldots,v_n\}$. The distance matrix $D(G)=[d_{ij}]$ of  graph $G$ is a real symmetric matrix of order $n\times n$, where $d_{ij}=d(v_i,v_j)$ be the length of the shortest path between vertices $v_i$ and $v_j$ and $d_{ii}=0$ for $1\leq i\leq n$.  The  Hadamard product  $D(G)\circ D(G)$ is called the squared distance matrix of $G$ and is denoted by $\Delta (G)$. That is, $\Delta (G)=[d_{ij}^2]$. Therefore, $\Delta (G)$ is a real symmetric matrix and all its eigenvalues are real. We write $\Delta$ instead of $\Delta (G)$ if there is no scope for confusion. 

In the study of distance matrices, results in~\cite{Gr1,Gr2} generated significant interest among
scholars. In~\cite{Gr1}, it was shown that the determinant of the distance matrix, $\det D(T)=(-1)^{n-1}(n-1)2^{n-2},$  holds for any tree $T$ with $n$ vertices, and this result is independent of the geometry of the tree. Subsequently, in~\cite{Gr2}, it was found that the inverse of $D(T)$ is a rank-one perturbation of the Laplacian matrix of the tree $T$. Given that the distance matrix and the Laplacian matrix appear to be unrelated, these findings have sparked considerable interest, leading to several extensions and generalisations (for example, see~\cite{Bp3,JD1, Hou2,Zhou1}).

Given a graph \( G \), the aim of these studies is to define a matrix \( \mathcal{L} \) such that both the row and column sums of \( \mathcal{L} \) are zero. Additionally, it is required that the inverse of the squared distance matrix be a rank-one perturbation of \( \mathcal{L} \). Since the row and column sums of the Laplacian matrix are also zero, the author in \cite{Bp3} referred to this matrix \( \mathcal{L} \) as a ``Laplacian-like" matrix. Similar extensions have been explored for certain classes of strongly connected directed graphs (for example see~\cite{Bp1,JD2,Hou1, Zhou2}). Another key focus in the study of distance matrices is to investigate its spectrum and related properties (for example, see~\cite{Ob,Ob1, Sun, So}).

The squared distance matrix was first studied for trees in \cite{Bp4}. To be specific, given a tree \( T \), a formula for \( \det \Delta(T) \) was derived in \cite{Bp4}. Subsequently, in \cite{Bp5}, the inverse of \(  \Delta(T) \) was determined (when it exists) as a rank-one perturbation of a matrix. In~\cite{Bp6}, the author examined the determinant and the inverse of the squared distance matrix of a weighted tree and later in~\cite{IM}, a similar study was done for a tree with matrix weights. Additionally, the determinant and inverse of the squared distance matrix for a complete multipartite graph were explored in \cite{JD3,JD4}.

Similar to the distance matrix, the spectral properties of the squared distance matrix are also explored. Since all the eigenvalues of a squared distance matrix are real, an interesting problem for researchers is to determine its inertia. Inertia of a real symmetric matrix is defined as a triplet that consists of the number of positive eigenvalues, the number of negative eigenvalues, and the multiplicity of zero as an eigenvalue. The inertia of the squared distance matrix for a tree was computed in \cite{Bp5}, and later, an alternative proof was provided in \cite{How}. In that same study, the authors also computed the inertia for a unicyclic graph. In another work~\cite{JD3}, the authors considered the complete multipartite graph \(K_{n_1,n_2,\ldots,n_t}\). They first calculated the inertia and energy of \(\Delta(K_{n_1,n_2,\ldots,n_t})\). Then, for fixed values of \(n\) and \(t\), they discussed the existence and uniqueness of graphs for which the spectral radius and energy of \(\Delta(K_{n_1,n_2,\ldots,n_t})\) reached their maximum and minimum values.

In this literature survey, we observed that while some results regarding distance matrices and square distance matrices may appear similar, the proofs of these results are quite distinct. In this article, we focus on studying the square distance matrix of a starlike block graph.  We begin with a formal definition of a starlike block graph: A vertex $v$ of a connected graph $G$ is a cut-vertex of $G$ if $G-v$ is disconnected. A block of the graph $G$ is a maximal connected subgraph of $G$ that has no cut-vertex. In literature, a simple connected graph is referred to as a block graph if each of its blocks is a complete graph. A block graph \( G \) is said to be a starlike block graph if it has a central cut-vertex. We will denote the starlike block graph with blocks \( K_{n_1+1}, K_{n_2+1}, \ldots, K_{n_b+1} \) as \( \mathcal{S}(n_1, n_2, \ldots, n_b) \), where the graph has a total of \( n=1 + \sum_{i=1}^b n_i \) vertices. 

Due to the result~\cite[Theorem~$9.5$]{Bapat}, the determinant and cofactor can be determined using their blocks in the context of the distance matrix of a graph. In a later study~\cite{Bp3}, the authors analysed block graphs and introduced a method to calculate the inverse of the distance matrix by first computing the inverse of its blocks, and then extending this to the entire block graph. Since then, many researchers have applied a similar methodology to calculate the inverse of the distance matrix for various class graphs. This approach applies to distance matrices due to a specific property of distances in a graph: If \( v_i \) and \( v_j \) are two distinct vertices that are not in the same block, and the shortest path between them passes through a cut vertex \( v \), then the distance between \( v_i \) and \( v_j \) can be expressed as \( d(v_i, v_j) = d(v_i, v) + d(v, v_j) \). It is easy to see that this property fails in the case of squared distance matrices. Based on the examples encountered during the preparation of this manuscript, in the context of a squared distance matrix, we need to look at the graph as a whole and the blocks in it. This article is an attempt in this direction.

In this article, we will calculate the characteristic polynomial, determinant, cofactor, and inertia using results related to equitable partitions of the matrix \(\Delta(\mathcal{S}(n_1, n_2, \ldots, n_b))\), along with one of its principal submatrices and a few other results from matrix theory. Next, we find a few recurrence-type relations that help us define the Laplacian-like matrix \(\mathcal{L}\) and compute the inverse of \(\Delta(\mathcal{S}(n_1, n_2, \ldots, n_b))\) in the desired form. Finally, we present an interesting observation regarding the characteristic polynomial of \(\Delta(\mathcal{S}(n_1, n_2, \ldots, n_b))\). Then, using an application of the Intermediate Value Theorem, we are able to identify the extremal graphs for which the spectral radius of \(\Delta(\mathcal{S}(n_1, n_2, \ldots, n_b))\) uniquely attains its maximum and minimum values, given \(n\) vertices and \(b\) blocks.

This article is organised as follows: Section~\ref{sec:matrix} presents several results from matrix theory and outlines the notations used throughout the article. In Section~\ref{sec:det-cof-iner}, we calculate the determinant and cofactor of \(\Delta(\mathcal{S}(n_1,n_2,\ldots,n_b))\). We also investigate its inertia. Section~\ref{sec:Inverse_Delta} introduces the Laplacian-like matrix $\mathcal{L}$ and demonstrates that the inverse of $\Delta(\mathcal{S}(n_1,n_2,\cdots,n_b))$ takes the desired form. Additionally, we prove that $\mathcal{L}$ is positive semidefinite with a rank of $n-1$. Finally, in Section~\ref{sec:spectral-radius}, we determine the extremal graphs for fixed values of \(n\) and \(b\), where the spectral radius of $\Delta(\mathcal{S}(n_1,n_2,\cdots,n_b))$ uniquely achieves its maximum and minimum values among all starlike graphs with \(n\) vertices and \(b\) blocks.

\subsection{Notations and a few results from Matrix Theory}\label{sec:matrix}

In this section, we first introduce a few notations and then recall a few results from matrix theory, which will be used time and again in this article.   Let $I_n$ and $ \mathds{1}_n$ denote the identity matrix and the column vector of all ones,  respectively.  We use $\mathbf{0}_{m \times n}$ to represent zero matrix of order $m \times n$. We write   $I$, $ \mathds{1}$  and $\mathbf{0}$  if there is no scope of confusion with respect to the order. Further, we write $J_{m \times n}$ to denote the $m\times n$ matrix of all ones, and if $m=n$, we use the notation $J_m$. Similarly, we will use $J$ if there is no confusion with its order.  We write $R_i \leftarrow R_i + k R_j$ to  represent the elementary row operation in which the row $R_i$ is replaced by the sum of $R_i$ and   $k$ times  $R_j$ for $1\leq i,j \leq n$. Similarly, we write $C_i \leftarrow C_i + k C_j$ for an analogous  elementary column operation.

Given a matrix $A$, we write $A^t$ to denote the transpose of a matrix $A$ and use the notation $A(i \mid j)$ to denote the submatrix obtained by deleting the $i^{th}$ row and the $j^{th}$ column. Let $A$ be an $n \times n$ matrix. For $1 \leq i,j \leq n$, the cofactor $c_{ij}$ is defined as $(-1)^{i+j} \det A(i \mid j)$. The transpose of the cofactor matrix $[c_{ij}]$  of  $A$ is called the adjoint of $A$, denoted by $\textup{ Adj }A$. We use the notation  $\cof A$ to denote the sum of all cofactors of $A$. Hence $\cof A=\mathds{1}_n^t \textup{ Adj }A \ \mathds{1}_n$. We now state a result that determines the $\cof A$  first using elementary row operations $R_i \leftarrow R_i - R_1$  for all $2\leq i \leq n$  and then the column operations $C_i \leftarrow C_i -  C_1$ for all $2\leq i \leq n$.

\begin{lem}\label{Lem:cof}\cite[Lemma~$8.4$]{Bapat}
Let $A$ be an $n \times n$ matrix. Let $M_A$ be the matrix obtained from $A$ by subtracting the first row from all other rows and then subtracting the first column from all other columns. Then $$\cof A = \det M_A(1 | 1).$$
\end{lem}

Given an $n\times n$ matrix $A$, we will use $P_A(x)$ to denote the characteristic polynomial of $A$. The set of eigenvalues of $A$ is called the spectrum of $A$,  denoted by $\sigma(A)$ and the spectral radius of $A$, denoted by $\rho(A)$ is defined as  $\ds \rho(A)= \max_{\lambda \in \sigma(A) } |\lambda|.$ By Perron-Frobenius theory, if $A$ is an entrywise nonnegative square matrix, then  $\rho(A)$ is the largest eigenvalue of $A$. Moreover, $\rho(A)$ is a simple eigenvalue of $A$ and is called the Perron value of $A$. Furthermore, if $A$ is an $n \times n$  real symmetric matrix,  we use the following convention where the eigenvalues of $A$ are in  decreasing order: 
\begin{equation}\label{eqn:ev-hermitian}
\lambda_{\max}=\lambda_1(A) \geq \lambda_2(A)\geq \cdots \geq \lambda_{n-1}(A)\geq  \lambda_{n}(A)=\lambda_{\min}.
\end{equation}
That is,  $\rho(A)=\lambda_1(A)$. We now state Weyl’s inequality which interlacing inequalities of a rank-one perturbation to a real symmetric matrix.
\begin{theorem}~\cite[Corollary~$4.3.9$]{Horn}\label{thm:interlacing}
Let $A$ and $B$ be real symmetric matrices of order $n \times n$ with eigenvalues ordered as in Eqn.~\eqref{eqn:ev-hermitian} such that $B = A+\mathbf{x}\mathbf{x}^t$, where $\mathbf{x}$ is a column vector. Then, 
\begin{equation*}
		\lambda_1(B) \geq \lambda_1(A) \geq \lambda_2(B) \geq \lambda_2(A) \geq \cdots \geq \lambda_n(B) \geq \lambda_n(A).
	\end{equation*}
\end{theorem}
We will now present a result concerning the spectral radius, which will be useful for the proofs in a later section.
\begin{prop}\label{prop:rank-one}
Let $B$ be an entrywise positive square matrix and $B=A+\mathbf{u}\mathbf{v}^t$, where $A$ is a diagonal matrix with negative entries and $\mathbf{u}, \mathbf{v}$ are column vectors. Then, $\rho(B)$, the spectral radius of $B$, is the only positive eigenvalue of $B$. 
\end{prop}
\begin{proof}
Let $B$ be an $n\times n$ matrix and $B=A+\mathbf{u}\mathbf{v}^t$, where $\mathbf{u}=(u_1,u_2,\ldots, u_n)^t$ and $\mathbf{v}=(v_1,v_2,\ldots, v_n)^t$. Since $B$ is an entrywise positive square matrix, and  $A$ is a diagonal matrix with negative entries, without loss of generality, we assume $u_i, v_i >0$ for all $i$. Let $\mathbf{x}$ and $\mathbf{y}$ be two column vectors of order $n$, where the $i^{th}$ entries of $\mathbf{x}$ and $\mathbf{y}$ are given by
$$x_i= \sqrt{u_iv_i} \mbox{ and } y_i=\sqrt{\frac{u_i}{v_i}},$$
respectively. Let $T=\textup{ diag} (\mathbf{y})^{-1} B \textup{ diag}(\mathbf{y})$. Then, $B$ and $T$ are similar matrices, and $T= A + \mathbf{x} \mathbf{x}^t$. Since all the eigenvalues of $A$ are negative, using Theorem~\ref{thm:interlacing}, $T$ has at least $n-1$ negative eigenvalues. Hence, $B$ has at least $n-1$ negative eigenvalues and  by Perron-Frobenius theory,  $\rho(B)$, the spectral radius of $B$ is an eigenvalue of $B$. This completes the proof.
\end{proof}

For a real  symmetric  matrix $A$, the inertia of  $A$ is  the triplet  $(\mathbf{n}_{+}(A),\mathbf{n}_{0}(A), \mathbf{n}_{-}(A) )$, denoted by $\textup{In}(A)$, where $\mathbf{n}_{+}(A),\mathbf{n}_{0}(A)  \mbox{ and } \mathbf{n}_{-}(A)$ denote the number of positive eigenvalues of $A$, the multiplicity of  $0$  as an eigenvalue of $A$ and the number of negative eigenvalues of $A$, respectively. Now we state a result known as Haynsworth inertia additivity formula.

\begin{prop}\label{prop:inertia}~\cite{Zhang1}
	Let $A$ be a real symmetric matrix partitioned as $$
	\begin{bmatrix}
		A_{11} & A_{12}\\
		A_{21} & A_{22}
	\end{bmatrix}
	$$ where $A_{22}$ is square and nonsingular. Then $$\textup{In}(A) = \textup{In}(A_{22}) + \textup{In}(A_{11} - A_{12}A_{22}^{-1}A_{21})$$
\end{prop}

We will now revisit the concepts of equitable partition and equitable quotient matrix. The equitable quotient matrix is recognised as a powerful tool for studying the spectrum of a matrix. Although the equitable quotient matrix is defined for a more general set-up, we will focus specifically on the results that are pertinent to this article. For a more general set-up, readers may refer~\cite{Horn,You}.

Suppose $A$ is a real symmetric  matrix whose rows and columns are indexed by
$X = \{1, . . . , n \}$. Let $\pi=\{X_1, . . . ,X_t \}$ be a partition of $X$.  Let $A$ be
partitioned according to $\{X_1, . . . ,X_t \}$, \emph{i.e.},
\begin{equation*}\label{eqn:M}
		A = \begin{bmatrix}
			A_{11}       &  \dots & A_{1t} \\
			\vdots & \ddots & \vdots\\
			A_{t1}       &  \dots & A_{tt}
		\end{bmatrix},
	\end{equation*}
where $A_{ij}$ denotes the submatrix (block) of $A$ formed by the rows in $X_i$ and the
columns in $X_j$. Let $q_{ij}(A)$ denote the average row sum of $A_{ij}$. 
Then, the matrix $\mathbf{Q}(A) = [q_{ij}(A)]$ is called the quotient matrix of $A $ with respect to a given partition $\pi$. 	If $A_{ij} \mathds{1} = q_{ij}(A) \mathds{1}$, {\it{i.e.,}} the row sum of each block $A_{ij}$ is constant  then the partition is called equitable, and $\mathbf{Q}(A)$ is called a equitable quotient matrix of $A$. The $n\times t$ matrix $S=[s_{ij}]$ is called the characteristic matrix with respect to the equatable partition  $\pi=\{X_1, . . . , X_t \}$, where $s_{ij}= 1$ if $i\in X_j$ and $0$ otherwise. 

We now state a theorem that gives a relation between the spectrum of a matrix and the spectrum of its quotient matrix. 
\begin{prop}\label{prop:quotient}~\cite{Horn,You}
Let $A$ be a real symmetric matrix and $\mathbf{Q}(A)$ be an equitable quotient matrix of $A$  with respect to a partition $\pi$.  Then, $\sigma (\mathbf{Q}(A)) \subset \sigma (A).$ Furthermore, the following holds.
\begin{enumerate}
\item[($i$)] If $S$ is the characteristic matrix with respect to $\pi$, and $(\lambda, \mathbf{x})$ is an eigenpair of $\mathbf{Q}(A)$, then  $(\lambda, S\mathbf{x})$ is an eigenpair of $A$.

\item[($ii$)] If $A$ is an entrywise nonnegative matrix, then  $\rho (A)= \rho(\mathbf{Q}(A))$.
\end{enumerate}

\end{prop}


\section{Determinant, Cofactor and Inertia of $ \Delta (\mathcal{S}(n_1,n_2,\cdots,n_b))$}\label{sec:det-cof-iner}
In this section, we calculate the determinant and cofactor of the squared distance matrix for a starlike block graph, and we also investigate its inertia. To achieve these goals, we will use results related to equitable partitions of matrices, specifically concerning the squared distance matrix, along with one of its principal submatrices for the starlike block graph and on a few related matrices.

Let $\mathcal{S}(n_1,n_2,\cdots,n_b)$ be the starlike block graph on $n$ vertices with blocks $K_{n_1+1},K_{n_2+1},\cdots,K_{n_b+1}$  with a central cut vertex. Then $n=1+\sum_{i=1}^b n_i$. Let $\pi=\{V_0,V_1,V_2,\cdots, V_b\}$ be a partition of vertex set of  $\mathcal{S}(n_1,n_2,\cdots,n_b)$, where $V_0$ is the singleton set consisting of the central cut vertex and for $1\leq i\leq b$, $V_i$ consisting of vertices of block $K_{n_i+1}$  except the central cut vertex. In view of the partition  $\pi$, the matrix $\Delta(\mathcal{S}(n_1,n_2,\cdots,n_b))$ can be expressed in the following block form:
\begin{equation}\label{eqn:Delta-starlike}
\Delta=\Delta(\mathcal{S}(n_1,n_2,\cdots,n_b)) = \left[
\begin{array}{c|c|c|c|c}
    0  & \mathds{1}_{n_1}^t& \mathds{1}_{n_2}^t & \cdots & \mathds{1}_{n_b}^t\\
    \midrule
	\mathds{1}_{n_1} & J_{n_1}-I_{n_1} & 4J_{n_1 \times n_2} & \cdots & 4J_{n_1 \times n_b}\\
	\midrule
	\mathds{1}_{n_2} &4J_{n_2 \times n_1} & J_{n_2}-I_{n_2}  & \cdots & 4J_{n_2 \times n_b}\\
	\midrule
	\vdots&\vdots&\vdots&\ddots&\vdots\\
	\midrule
	\mathds{1}_{n_b} &4J_{n_b \times n_1} & 4J_{n_b \times n_2}  & \cdots & J_{n_b}-I_{n_b} \\
\end{array} \right].
\end{equation}
Another block matrix representation of the squared distance matrix $\Delta= 
	\begin{bmatrix}
		\Delta_{11} & \Delta_{12}\\
		\Delta_{21} & \Delta_{22}
	\end{bmatrix},
		$
		where $\Delta_{11}=0$, $\Delta_{12}^t=\Delta_{21}=\mathds{1}_{n-1}$ and 
\begin{equation}\label{eqn:Delta22}
\Delta_{22}= \left[
\begin{array}{c|c|c|c}
   	 J_{n_1}-I_{n_1} & 4J_{n_1 \times n_2} & \cdots & 4J_{n_1 \times n_b}\\
	\midrule
	4J_{n_2 \times n_1} & J_{n_2}-I_{n_2}  & \cdots & 4J_{n_2 \times n_b}\\
	\midrule
	\vdots&\vdots&\ddots&\vdots\\
	\midrule
	4J_{n_b \times n_1} & 4J_{n_b \times n_2}  & \cdots & J_{n_b}-I_{n_b} \\
\end{array} \right].
\end{equation}
Thus, $\Delta_{22}$ is the principal submatrix of $\Delta$, and the block matrix form of $\Delta_{22}$ is with respect to the vertex partition  $\widetilde{ \pi}=\{V_1,V_2,\cdots, V_b\}$. 
The quotient matrices of  $\Delta$ and $\Delta_{22}$ with respect to the partitions  $\pi=\{V_0,V_1,V_2,\cdots, V_b\}$ and  $\widetilde{ \pi}=\{V_1,V_2,\cdots, V_b\}$ are given by
\begin{equation}\label{eqn:Q_Delta}
	\mathbf{Q}(\Delta) = \left[
	\begin{array}{c|cccc}
		0 & n_1 & n_2 & \cdots & n_b\\
		\midrule
		1 & n_1-1 & 4n_2 & \cdots & 4n_b\\
		1 & 4n_1 & n_2-1 & \cdots & 4n_b\\
		\vdots&\vdots&\vdots&\ddots&\vdots\\
		1 & 4n_1 & 4n_2 & \cdots & n_b-1
	\end{array} \right]=
	\left[
\begin{array}{c|c c}
0 & n_1 \ n_2 \dots n_b    \\ \hline
1 &  \\  
\vdots  &   \mathbf{Q}(\Delta_{22})  \\
1   &\\
1   &\\
\end{array}
\right].
\end{equation}

We will organise the results of this section into three subsections. In the first two subsections, we will calculate the determinant and cofactor for both \(\Delta\) and \(\Delta_{22}\). Finally, in the last subsection, we will use the results for \(\det \Delta_{22}\) and \(\cof \Delta_{22}\) to obtain  the inertia of \(\Delta\).

\subsection{Determinant of $\Delta (\mathcal{S}(n_1,n_2,\cdots,n_b))$}
In this section, we will computes the $\det \Delta$ and $\det \Delta_{22}.$ We begin with a lemma that compute the characteristic polynomial of the quotient matrices  $\mathbf{Q}(\Delta)$ and $\mathbf{Q}(\Delta_{22})$.

\begin{lem}\label{lem:char-Q-Delta}
Let $\Delta$ be the squared distance matrix of the starlike block graph $\mathcal{S}(n_1,n_2,\cdots,n_b)$ and $\Delta_{22}$ be the principal submatrix of $\Delta$ as defined in Eqn.~\eqref{eqn:Delta22}. If $\mathbf{Q}(\Delta)$ and $\mathbf{Q}(\Delta_{22})$ is the quotient matrices of $\Delta$ and $\Delta_{22}$ as defined in Eqn.~\eqref{eqn:Q_Delta}, then the characteristic polynomial of the $\mathbf{Q}(\Delta)$ and $\mathbf{Q}(\Delta_{22})$ are given by
$$ P_{\mathbf{Q}(\Delta)}(x)= x\prod_{i=1}^{b}(x+3n_i+1) -(4x+1)\sum_{i=1}^b n_i \prod_{j \ne i} (x+3n_i+1),$$
and 
$$ P_{\mathbf{Q}(\Delta_{22})}(x)=\prod_{i =1}^{b}(x+3n_i+1) - 4\sum_{i =1}^{b} n_i \prod_{j \ne i} (x+3n_j+1).$$
\end{lem}
\begin{proof}
We use elementary row operations on matrices $xI-\mathbf{Q}(\Delta)$ and $xI-\mathbf{Q}(\Delta_{22})$ to compute the characteristic polynomial of $\mathbf{Q}(\Delta)$ and $\mathbf{Q}(\Delta_{22})$, respectively. Note that,
$$
xI - \mathbf{Q}(\Delta) = \left[
\begin{array}{c|cccc}
	x & -n_1 & -n_2 & \cdots & -n_b\\
	\midrule
	-1 & x-n_1+1 & -4n_2 & \cdots & -4n_b\\
	-1 & -4n_1 & x-n_2+1 & \cdots & -4n_b\\
	\vdots&\vdots&\vdots&\ddots&\vdots\\
	-1 & -4n_1 & -4n_2 & \cdots & x-n_b+1
\end{array} \right]= 
\left[
\begin{array}{c|c c}
x & -n_1 \ -n_2 \dots -n_b    \\ \hline
-1 &  \\  
\vdots  &   xI-\mathbf{Q}(\Delta_{22})  \\
-1   &\\
-1   &\\
\end{array}
\right]
$$
For each $2\leq i\leq b+1,$ applying the elementary row operation $R_i \leftarrow R_i -4 R_1$  on $xI-\mathbf{Q}(\Delta)$, the resulting matrix is given by
\begin{equation*}
\left[ \begin{array}{ccccc}
	x & -n_1 & -n_2 & \cdots & -n_b\\
	-4x-1 & x+3n_1+1 & 0 & \cdots & 0\\
	-4x-1 & 0 & x+3n_2+1 & \cdots & 0\\
	\vdots&\vdots&\vdots&\ddots&\vdots\\
	-4x-1 & 0 & 0 & \cdots & x+3n_b+1
\end{array} \right],
\end{equation*} and expanding along the first row gives
$$
\det(xI - \mathbf{Q}(\Delta)) = x\prod_{i=1}^{b}(x+3n_i+1) -(4x+1)\sum_{i=1}^b n_i \prod_{j \ne i} (x+3n_i+1).$$

Next, applying the elementary row operation $R_i \leftarrow R_i - R_1$ for each $2\leq i\leq b,$  on $xI-\mathbf{Q}(\Delta_{22})$,  the resulting matrix is 

\begin{equation*}
\left[ \begin{array}{cccc}
		(x+3n_1+1)-4n_1 & -4n_2 & \cdots & -4n_b\\
	 -(x+3n_1+1) & x+3n_2+1 & \cdots & 0\\
	\vdots&\vdots&\ddots&\vdots\\
	 -(x+3n_1+1)& 0 & \cdots & x+3n_b+1
\end{array} \right],
\end{equation*} and expanding along the first row gives
$$\det(xI - \mathbf{Q}(\Delta_{22}))=\prod_{i =1}^{b}(x+3n_i+1) - 4\sum_{i =1}^{b} n_i \prod_{j \ne i} (x+3n_j+1). $$
This completes the proof.
\end{proof}

The following result indicates that \(-1\) is an eigenvalue for both \(\Delta\) and \(\Delta_{22}\).

\begin{lem}\label{lem:ev_Delta_minus1}
Let $\Delta$ be the squared distance matrix of the starlike block graph $\mathcal{S}(n_1,n_2,\cdots,n_b)$ and $\Delta_{22}$ be the principal submatrix of $\Delta$ as defined in Eqn.~\eqref{eqn:Delta22}. Then, $-1$ is an eigenvalue of both the matrices $\Delta$ and $\Delta_{22}$ with multiplicity at least $\ds \sum_{i=1}^b n_i -b.$ 
\end{lem}
\begin{proof}
Let $\mathbf{e}(i,j)$ be a column vector of order $\sum_{i=1}^b n_i$ whose $i^{th}$ entry is $1$ and $j^{th}$ entry is $-1$.  Let
	\begin{align*}
		\mathcal{E} = \{\mathbf{e}(1,j) \mid j=2,\cdots,n_1\} & \cup \{\mathbf{e}(n_1+1,n_1+j) \mid j=2,\cdots,n_2\}\\
		& \cup \cdots \cup \{\mathbf{e} \left(\sum_{k=1}^{b-1}n_k+1,\sum_{k=1}^{b-1}n_k+j \right) \mid j=2,\cdots,n_b\}.
	\end{align*}
Note that, if  $\mathbf{e}(i,j)$ is a column vector of conformal order as $J-I$, then $\mathbf{e}(i,j)$ is an eigenvector of $J-I$ corresponds to the eigenvalue $-1.$ Therefore, $\Delta_{22}\mathbf{x}= (-1)\mathbf{x}$ for all $\mathbf{x} \in \mathcal{E}$. Since  each set in the union in $\mathcal{E}$ corresponds to the partition $\widetilde{ \pi}=\{V_1,V_2,\cdots, V_b\}$ and hence the cardinality of the set $\mathcal{E}$ is $| \mathcal{E}|= \sum_{i=1}^b n_i -b.$ Therefore, $-1$ is eigenvalue of $\Delta_{22}$ with multiplicity at least $ \sum_{i=1}^b n_i -b.$ 

Recall,  $\Delta=\left[\begin{array}{c|c}
		0 & \mathds{1}^t\\
		\midrule
		\mathds{1}& \Delta_{22}
		\end{array}	\right],
		$ and hence for all $\mathbf{x} \in \mathcal{E}$, we have
		$$\left[\begin{array}{c|c}
		0 & \mathds{1}^t\\
		\midrule
		\mathds{1}& \Delta_{22}
		\end{array}	\right]	
		\left[\begin{array}{c}
		0\\
		\midrule
		\mathbf{x}
		\end{array}	\right]=
		\left[\begin{array}{c}
		 \mathds{1}^t \mathbf{x}\\
		\midrule
		\Delta_{22} \mathbf{x}
		\end{array}	\right]	=
		\left[\begin{array}{c}
		 0\\
		\midrule
		- \mathbf{x}
		\end{array}	\right]=
		(-1) \left[\begin{array}{c}
		 0\\
		\midrule
		\mathbf{x}
		\end{array}	\right].		
		$$
This concludes the proof.
\end{proof}

\begin{rem}
From the proof of Lemma~\ref{lem:ev_Delta_minus1}, it is evident that the contribution to the multiplicity of the eigenvalue \(-1\) is \(n_i - 1\) for each block \(K_{n_i + 1}\), where \(1 \leq i \leq b\). Thus, if any of the blocks is \(K_2\), the contribution becomes \(0\). It is also reflected in the proof of Lemma~\ref{lem:ev_Delta_minus1}, where the set corresponding to the block \(K_2\) in the set \(\mathcal{E}\) is an empty set. Therefore, the Lemma~\ref{lem:ev_Delta_minus1} holds regardless of the sizes of the blocks.
\end{rem}

We now state and prove a theorem that determines the characteristic polynomial $\Delta$ and $\Delta_{22}$.

\begin{theorem}\label{thm:char-poly}
Let $\Delta$ be the squared distance matrix of the starlike block graph $\mathcal{S}(n_1,n_2,\cdots,n_b)$ and $\Delta_{22}$ be the principal submatrix of $\Delta$ as defined in Eqn.~\eqref{eqn:Delta22}. Then, the characteristic polynomial of $\Delta$ and $\Delta_{22}$ are given by
$$
P_{\Delta}(x)=	(x+1)^{\sum_{i=1}^{b}n_i - b}  \left(x\prod_{i=1}^{b}(x+3n_i+1) -(4x+1)\sum_{i=1}^b n_i \prod_{j \ne i} (x+3n_i+1)\right),
	$$
	and
$$P_{\Delta_{22}}(x)=	(x+1)^{\sum_{i=1}^{b}n_i - b}  \left(\prod_{i =1}^{b}(x+3n_i+1) - 4\sum_{i =1}^{b} n_i \prod_{j \ne i} (x+3n_j+1)\right).$$
	
\end{theorem}
\begin{proof}
Using Lemma~\ref{lem:char-Q-Delta}, we observe that since \( P_{\mathbf{Q}(\Delta)}(-1) \neq 0 \), it follows that \(-1\) is not an eigenvalue of the quotient matrix \( \mathbf{Q}(\Delta) \). Using Theorem~\ref{prop:quotient}, \( \sigma(\mathbf{Q}(\Delta)) \subset \sigma(\Delta) \). This implies that \(\Delta\) has \( b+1 = n - \left(\sum_{i=1}^b n_i - b\right) \) eigenvalues other than \(-1\). Additionally, from Lemma~\ref{lem:ev_Delta_minus1}, we know that \(-1\) is an eigenvalue of \(\Delta\) with a multiplicity of at least \(\sum_{i=1}^b n_i- b\). Therefore, the multiplicity of the eigenvalue \(-1\) is exactly \(\sum_{i=1}^b n_i - b\), and  
$$
P_{\Delta}(x)=	(x+1)^{\sum_{i=1}^{b}n_i - b}  \left(x\prod_{i=1}^{b}(x+3n_i+1) -(4x+1)\sum_{i=1}^b n_i \prod_{j \ne i} (x+3n_i+1)\right).
$$
Analogous argument also gives $-1$ is an eigenvalue of  $\Delta_{22}$ with multiplicity $\ds \sum_{i=1}^b n_i -b,$ and 
$$P_{\Delta_{22}}(x)=	(x+1)^{\sum_{i=1}^{b}n_i - b}  \left(\prod_{i =1}^{b}(x+3n_i+1) - 4\sum_{i =1}^{b} n_i \prod_{j \ne i} (x+3n_j+1)\right).$$
\end{proof}

We now present the corollary stating the determinant of $\Delta$ and $\Delta_{22}$, which can be obtained by substituting $x=0$ into the characteristic polynomials $P_{\Delta}(x)$ and $P_{\Delta_{22}}(x)$.

\begin{cor}\label{cor:det}
Let $\Delta$ be the squared distance matrix of the starlike block graph $\mathcal{S}(n_1,n_2,\cdots,n_b)$ and $\Delta_{22}$ be the principal submatrix of $\Delta$ as defined in Eqn.~\eqref{eqn:Delta22}. then, the determinant of $\Delta$ and $\Delta_{22}$ are given by
	$$\det \Delta = (-1)^{\sum_{i=1}^b n_i} \sum_{i=1}^b n_i \prod_{j \ne i} (3n_i+1),$$
	and
		$$
	\det \Delta_{22} = (-1)^{\sum_{i =1}^{b} n_i -1} \left(4\sum_{i =1}^{b} n_i \prod_{j \ne i} (3n_j+1) - \prod_{i =1}^{b}(3n_i+1)\right).
		$$

\end{cor}

\subsection{Cofactor of $\Delta (\mathcal{S}(n_1,n_2,\cdots,n_b))$}

In this section, we will compute the $\cof \Delta$ and $\cof \Delta_{22}.$ From Lemma~\ref{Lem:cof}, given a matrix $A$, the $\cof A= \det M_{A}(1|1)$, where $M_{A}(1|1)$ is the matrix obtained from $A$ by subtracting the first row from all other rows and then subtracting the first column from all other columns. Thus, from Eqn.~\eqref{eqn:Delta-starlike} and~\eqref{eqn:Delta22}, the block matrix form of $M_{\Delta}(1|1)$ and $M_{\Delta_{22}}(1|1)$ are given by
\begin{equation}\label{eqn:M_cof_Delta-1}
	M_{\Delta}(1|1) = \left[
	\begin{array}{c|c|c|c}
		-J_{n_1}-I_{n_1} & 2J_{n_1 \times n_2} & \cdots & 2J_{n_1 \times n_b}\\
		\midrule
		2J_{n_2 \times n_1} & -J_{n_2}-I_{n_2}  & \cdots & 2J_{n_2 \times n_b}\\
		\midrule
		\vdots&\vdots&\ddots&\vdots\\
		\midrule
		2J_{n_b \times n_1} & 2J_{n_b \times n_2}  & \cdots & -J_{n_b}-I_{n_b} \\
	\end{array} \right],
\end{equation}
and 
\begin{equation}\label{eqn:M_cof_Delta-22}
	M_{\Delta_{22}}(1|1) = \left[
	\begin{array}{c|c|c|c}
		-J_{n_1-1}-I_{n_1-1} & 2J_{n_1 \times n_2} & \cdots & 2J_{n_1 \times n_b}\\
		\midrule
		2J_{n_2 \times n_1} & -7J_{n_2}-I_{n_2}  & \cdots & 2J_{n_2 \times n_b}\\
		\midrule
		\vdots&\vdots&\ddots&\vdots\\
		\midrule
		2J_{n_b \times n_1} & 2J_{n_b \times n_2}  & \cdots & -7J_{n_b}-I_{n_b} \\
	\end{array} \right].
\end{equation}

The proof of the following lemma is similar to Lemma~\ref{lem:ev_Delta_minus1}, so we state the result without proof.

\begin{lem}\label{lem:cof_ev_Delta_minus1}
Let $M_{\Delta}(1|1)$ and $M_{\Delta_{22}}(1|1)$ be the matrices defined in Eqns.~\eqref{eqn:M_cof_Delta-1}and~\eqref{eqn:M_cof_Delta-22}. Then, $-1$ is an eigenvalue of both the matrices $M_{\Delta}(1|1)$ and $M_{\Delta_{22}}(1|1)$ with multiplicity at least $\ds \sum_{i=1}^b n_i -b$ and  $\ds \sum_{i=1}^b n_i -(b+1),$ respectively.
\end{lem}

We will now calculate the cofactors \(\cof \Delta\) and \(\cof \Delta_{22}\) in the following result.

\begin{theorem}\label{thm:cof}
Let $\Delta$ be the squared distance matrix of the starlike block graph $\mathcal{S}(n_1,n_2,\cdots,n_b)$ and $\Delta_{22}$ be the principal submatrix of $\Delta$ as defined in Eqn.~\eqref{eqn:Delta22}.  Then, $\cof \Delta$ and $\cof \Delta_{22}$ are given by
\begin{equation*}\label{eqn:cof_Delta}
\cof \Delta = (-1)^{\sum_{i=1}^{b}n_i }\left[\prod_{i=1}^b(3n_i+1) -2\sum_{i=1}^b n_i \prod_{j \ne i}(3n_j+1)\right],
\end{equation*}
and 
$$
	\cof \Delta_{22} = (-1)^{\sum_{i =1}^b n_i - 1} \left(\sum_{i =1 }^b n_i \prod_{j \ne i} (3n_j+1) \right).
	$$
\end{theorem}

\begin{proof}

The quotient matrix of $M_{\Delta}(1|1)$ as defined in Eqn.~\eqref{eqn:M_cof_Delta-1} is:
\begin{equation*}\label{eqn:Q(M)}
\mathbf{Q}(M_{\Delta}(1|1)) = \left[
\begin{array}{cccc}
	-n_1-1 & 2n_1 &\cdots& 2n_1\\
	2n_2 & -n_2-1  &\cdots& 2n_2\\
	\vdots&\vdots&\ddots&\vdots\\
	2n_b & 2n_b &\cdots & -n_b-1 \\
\end{array}\right].
\end{equation*}
It is straightforward to observe that if \( \mathbf{Q}(M_{\Delta}(1|1)) \mathbf{x} = 0 \), then \( \mathbf{x} = 0 \). Therefore, \( -1 \) is not an eigenvalue of \( \mathbf{Q}(M_{\Delta}(1|1)) \). By Theorem \ref{prop:quotient}, it follows that \( \sigma(\mathbf{Q}(M_{\Delta}(1|1))) \subset \sigma(M_{\Delta}(1|1)) \). Consequently, \( M_{\Delta}(1|1) \) has \( b \) eigenvalues other than \( -1 \), and the multiplicity of \( -1 \) is given by \( \sum_{i=1}^b n_i - b \). Thus,  $$\cof \Delta= (-1)^{ \sum_{i=1}^b n_i -b} \det \mathbf{Q}(M_{\Delta}(1|1)).$$

Next, we apply elementary column operations \( C_i \leftarrow C_i - C_1 \) for all \( 2 \leq i \leq b \) on \( \mathbf{Q}(M_{\Delta}(1|1)) \), resulting in: 
 $$\left[
	\begin{array}{cccc}
		-n_1-1 & 3n_1+1 &\cdots& 3n_1+1\\
		2n_2 & -3n_2-1  &\cdots& 0\\
		\vdots&\vdots&\ddots&\vdots\\
		2n_b & 0 &\cdots & -3n_b-1 \\
	\end{array}\right]$$ and by expanding along the first row, we have
	$$\det \mathbf{Q}(M(1|1)) = (-1)^b \left[\prod_{i=1}^b(3n_i+1) -2\sum_{i=1}^b n_i \prod_{j \ne i}(3n_j+1)\right].$$
This completes the proof for $\cof \Delta.$
	
Now, we consider the quotient matrix of \( M_{\Delta_{22}}(1|1) \) as defined in Eqn.~\eqref{eqn:M_cof_Delta-22}, given by
\begin{equation*}\label{eqn:Q(M)}
\mathbf{Q}(M_{\Delta_{22}}(1|1)) = \left[
\begin{array}{cccc}
	-n_1 & -n_2 & \cdots & n_b\\
	1-n_1 & -7n_2-1  &\cdots& -n_b\\
	\vdots&\vdots&\ddots&\vdots\\
	1-n_1 & -n_2 &\cdots & -7n_b-1 \\
\end{array}\right].
\end{equation*}

We perform elementary row operations \( R_i \leftarrow R_i - 4R_1 \) for all \( 2 \leq i \leq b \), leading to:
 $$\left[
	\begin{array}{cccc}
		-n_1 & -n_2 &\cdots& n_b\\
		1+3n_1 & -3n_2-1  &\cdots& 0\\
		\vdots&\vdots&\ddots&\vdots\\
		1+3n_b & 0 &\cdots & -3n_b-1 \\
	\end{array}\right]$$ 

Then, expanding along the first row yields:
\[
\det \mathbf{Q}(M_{\Delta_{22}}(1|1)) = (-1)^b \left(\sum_{i = 1}^b n_i \prod_{j \ne i} (3n_j + 1) \right).
\]

Therefore, using arguments similar to those used in computing the \( \cof \Delta \), the desired  result for \( \cof \Delta_{22} \) follows.
\end{proof}

The following corollary provides the necessary and sufficient conditions for a starlike block graph where $\cof \Delta=0$.

\begin{cor}\label{cor:cof}
Let $G=\mathcal{S}(n_1,n_2,\cdots,n_b)$ and $\Delta(G)$ be the squared distance matrix of $G$. Then,  $\cof \Delta(G)=0$ if and only if $G=\mathcal{S}(1,1)$, a path of length $2.$
\end{cor}
\begin{proof}
Let $G=\mathcal{S}(n_1,n_2,\cdots,n_b)$ and $\cof \Delta(G)=0$. By Theorem~\ref{thm:cof}, we get 
$$\ds \prod_{i=1}^b(3n_i+1) = 2\sum_{i=1}^b n_i \prod_{j \ne i}(3n_j+1)$$ which implies that $\ds \sum_{i=1}^b \frac{n_i}{3n_i+1} = \frac{1}{2}.$ Since $n_i \geq 1$, we have $ \dfrac{n_i}{3n_i+1}\geq \dfrac{1}{4}$. Therefore, $\ds \sum_{i=1}^b \frac{n_i}{3n_i+1} = \frac{1}{2}$ if and only if $b=2$ and  $n_1=n_2=1$. Hence, the desired result follows.
\end{proof}

\subsection{Inertia of $\Delta(\mathcal{S}(n_1,n_2,\cdots,n_b))$}

In this section, we first determine the inertia of $\Delta_{22}$, and subsequently use this result to determine the inertia of $\Delta$, the squared distance matrix of the starlike block graph $\mathcal{S}(n_1,n_2,\cdots,n_b)$.  The lemma below computes the inertia of $\Delta_{22}$.

\begin{lem}\label{lem:Q(Delta_22)}
Let $\Delta$ be the squared distance matrix of the starlike block graph $\mathcal{S}(n_1,n_2,\cdots,n_b)$ on  $n = 1+ \sum_{i=1}^b n_i$ vertices and $\Delta_{22}$ be the principal submatrix of $\Delta$ as defined in Eqn.~\eqref{eqn:Delta22}. If $\mathbf{Q}(\Delta_{22})$ is the quotient matrix of $\Delta_{22}$ as defined in Eqn~\eqref{eqn:Q_Delta},then, the following result holds:
 
\begin{enumerate}
\item [($i$)]  $\rho(\mathbf{Q}(\Delta_{22}))$, the spectral radius of $\mathbf{Q}(\Delta_{22})$ is the only positive eigenvalue of  $\mathbf{Q}(\Delta_{22})$. 

\item [($ii$)] $\Delta_{22}$ is a nonsingular matrix. 

\item [($iii$)] The inertia of $\Delta_{22}$, $\textup{In}(\Delta_{22})=(1, 0, n-2).$
\end{enumerate}

\end{lem}
\begin{proof}
The matrix $\mathbf{Q}(\Delta_{22})$ can be expressed as:
	\begin{equation}
		\mathbf{Q}(\Delta_{22}) =  \left[
		\begin{array}{cccc}
			-3n_1-1 & 0 & \cdots & 0\\
			0 & -3n_2-1 & \cdots & 0\\
			\vdots&\vdots&\ddots&\vdots\\
			0 & 0 & \cdots & -3n_b-1
		\end{array} \right] + \left[
		\begin{array}{cccc}
			4n_1 & 4n_2 & \cdots & 4n_b\\
			4n_1 & 4n_2 & \cdots & 4n_b\\
			\vdots&\vdots&\ddots&\vdots\\
			4n_1 & 4n_2 & \cdots & 4n_b
		\end{array} \right].
	\end{equation} 
That is, $\mathbf{Q}(\Delta_{22})$ is a rank-one perturbation of a diagonal matrix with negative entries. Therefore, the part~$(i)$ follows from Proposition~\ref{prop:rank-one}.

Next, observe that 
$$4\sum_{i =1}^{b} n_i \prod_{j \ne i} (3n_j+1) - \prod_{i =1}^{b}(3n_i+1)=  (n_1-1) \prod_{j = 2}^b (3n_j+1) + 4\sum_{i =2}^{b} n_i \prod_{j \ne i} (3n_j+1),$$ 
and hence by Corollary~\ref{cor:det}, we have $\det \Delta_{22} \neq 0$. This proves part~$(ii)$.

We now proceed to prove the part~$(iii)$. From the proof of Theorem~\ref{thm:char-poly}, it is clear that $-1$ is an eigenvalue of $\Delta_{22}$ with multiplicity $(n-1)-b=\sum_{i=1}^b n_i -b$, and the other $b$ eigenvalues are the eigenvalues of $\mathbf{Q}(\Delta_{22})$. Using part~$(i)$ and~$(ii)$, $\mathbf{Q}(\Delta_{22})$ has exactly one positive eigenvalue and $b-1$ negative eigenvalue. Therefore, $\textup{In}(\Delta_{22})=(1, 0, n-2)$. 
\end{proof}

We conclude this section with the theorem that computes the inertia of $\Delta$.
\begin{theorem}\label{thm:in_sq}
Let $\Delta$ be the squared distance matrix of the starlike block graph $\mathcal{S}(n_1,n_2,\cdots,n_b)$ on  $n = 1+ \sum_{i=1}^b n_i$ vertices.  Then, the inertia of  $\Delta$ is given by
	$\textup{In} (\Delta) = (1,0,n-1).$
\end{theorem}
\begin{proof}
Recall, the block matrix representation of the squared distance matrix $\Delta= 
	\begin{bmatrix}
		\Delta_{11} & \Delta_{12}\\
		\Delta_{21} & \Delta_{22}
	\end{bmatrix},
		$
		where $\Delta_{11}=0$, $\Delta_{12}^t=\Delta_{21}=\mathds{1}$ and $\Delta_{22}$ as defined in 	Eqn.\eqref{eqn:Delta22}. In view of part~$(ii)$ of Lemma~\ref{lem:Q(Delta_22)} and Proposition~\ref{prop:inertia}, we get
$$\textup{In}(\Delta) = \textup{In}(\Delta_{22}) + \textup{In}(\Delta_{11}-\Delta_{12}\Delta_{22}^{-1}\Delta_{21}).$$		
Using $\Delta_{11}=0$ and $\Delta_{12}^t=\Delta_{21}=\mathds{1}$, we have
 $\Delta_{11}-\Delta_{12}\Delta_{22}^{-1}\Delta_{21}= -\mathds{1}^t \Delta_{22}^{-1}\mathds{1}= - \dfrac{\cof \Delta_{22}}{\det \Delta_{22}}.$	From Corollary~\ref{cor:det} and Theorem~\ref{thm:cof}, we get $ \dfrac{\cof \Delta_{22}}{\det \Delta_{22}} >0$, and hence $\textup{In}(- \dfrac{\cof \Delta_{22}}{\det \Delta_{22}})= (0,0,1)$. Therefore, using part~$(iii)$ of Lemma~\ref{lem:Q(Delta_22)}, we have  
  $$\textup{In}(\Delta) = \textup{In}(\Delta_{22}) + \textup{In}( - \dfrac{\cof \Delta_{22}}{\det \Delta_{22}})=(1,0,n-2)+(0,0,1)=(1,0,n-1).$$
This concludes the proof.
\end{proof}

\section{Inverse of $\Delta (\mathcal{S}(n_1,n_2,\cdots,n_b))$}\label{sec:Inverse_Delta}

In this section, we will find the inverse of the matrix \(\Delta = \Delta(\mathcal{S}(n_1, n_2, \ldots, n_b))\) as a rank-one perturbation of a Laplacian-like matrix. According to Corollary~\ref{cor:det}, we have \(\det \Delta \neq 0\). Furthermore, referring to ~\cite[Lemma~$4.13$]{JD1}, the desired representation of the inverse \(\Delta^{-1}\) can be achieved only if \(\cof \Delta \neq 0\). Thus, based on Lemma~\ref{cor:cof}, the findings in this section hold  for all starlike block graphs except for \(\mathcal{S}(1,1)\), which is a path of length 2.

We now introduce a few notations that facilitate us in computing the inverse of $\Delta$. Let $n_i \in \mathbb{N}$ for $1\leq i\leq b$ and $b\geq 2$, let us denote
\begin{equation}\label{eqn:alpha}
\begin{cases}
\alpha=\alpha_{n_1,n_2,\cdots,n_b} = \ds \prod_{k=1}^b (3n_k+1),\\

\alpha_{\widehat{n_i}} =\ds \prod_{k=1\atop k \neq i}^b (3n_k+1) \text{ and } \alpha_{\widehat{n_in_j}} =\ds \prod_{k=1\atop k \neq i,j}^b (3n_k+1).\\
\end{cases}
\end{equation}
and
\begin{equation}\label{eqn:beta}
\quad \ 
\begin{cases}
\beta=\beta_{n_1,n_2,\cdots,n_b} = \ds \sum_{k=1}^b n_k \prod_{j=1 \atop j\neq k}^b (3n_j+1)=  \sum_{k=1}^b n_k  \alpha_{\widehat{n_k}},\\
\beta_{\widehat{n_i}} =\ds \sum_{k=1\atop k\neq i}^b n_k \prod_{j=1 \atop k \neq i, j}^b (3n_j+1)=  \sum_{k=1 \atop k\neq i}^b n_k  \alpha_{\widehat{n_i n_k}}.
\end{cases}
\end{equation}
Let $\mathcal{S}(n_1,n_2,\cdots,n_b)$ be the starlike block graph on $n$ vertices with blocks $K_{n_1+1},K_{n_2+1},\cdots,K_{n_b+1}$  with a central cut vertex. Then $n=1+\sum_{i=1}^b n_i$. 
Recall that,  $\pi=\{V_0,V_1,V_2,\cdots, V_b\}$ is the partition of vertex set of  $\mathcal{S}(n_1,n_2,\cdots,n_b)$, where $V_0$ is the singleton set consisting of the central cut vertex and for $1\leq i\leq b$, $V_i$ consisting of vertices of block $K_{n_i+1}$  except the central cut vertex. We will use Eqns.~\eqref{eqn:alpha} and~\eqref{eqn:beta}, and a few lemmas to compute the inverse of $\Delta$.

Let $\mathcal{\widehat{L}}=[\mathcal{\widehat{L}}_{uv}]$ be a symmetric matrix of order $n\times n$, where
\begin{equation}\label{eqn:L-hat}
	\mathcal{\widehat{L}}_{uv} = \begin{cases}
		\beta & \text{ if } u=v \text{ and } u\in V_0,\\
		- \alpha_{\widehat{n_i}} & \text{ if } u \neq v, u\in V_0 \text{ and } v\in V_i, \\
		(6\beta_{\widehat{n_i}} -\alpha_{\widehat{n_i}}) +(\alpha-2\beta) & \text{ if } u=v \text{ and } u\in V_i, \\
		(6\beta_{\widehat{n_i}} -\alpha_{\widehat{n_i}})  & \text{ if }  u \neq v \text{ and } u,v\in V_i \text{ and } \\
		2\alpha_{\widehat{n_in_j}} & \text{ if }  u \neq v,  u\in V_i \text{ and } v \in V_j.\\
	\end{cases}
\end{equation}

\begin{lem}\label{lem:L-hat1=0}
 $\mathcal{\widehat{L}}\mathds{1} = \mathbf{0}$ and $\mathds{1}^t \mathcal{\widehat{L}}= \mathbf{0}$ 
\end{lem}
\begin{proof}
Let $\nu = \mathcal{\widehat{L}}\mathds{1}$. For $v\in V_0$, we get $\nu(v)=\beta - \sum_{k=1}^b n_k \alpha_{\widehat{n_k}} =\beta -\beta=0.$ Next, for $1\leq i \leq b$ and  $v\in V_i$, we have
\begin{align}\label{eqn:L.1=0}
\nu(v)&= - \alpha_{\widehat{n_i}} + n_i (6\beta_{\widehat{n_i}} -\alpha_{\widehat{n_i}}) +(\alpha-2\beta) + 2 \sum_{k=1\atop k\neq i}^{b}n_k \alpha_{\widehat{n_in_k}}\\ \nonumber
&=  - \alpha_{\widehat{n_i}} + [  n_i (6\beta_{\widehat{n_i}} -\alpha_{\widehat{n_i}}) + (\alpha-2\beta)+ 2 \beta_{\widehat{n_i}} ]\\ \nonumber
&=  - \alpha_{\widehat{n_i}} + [(\alpha - n_i \alpha_{\widehat{n_i}}) + 2[(3n_i+1)\beta_{\widehat{n_i}} - \beta]]\\ \nonumber
&= - \alpha_{\widehat{n_i}} + [(2 n_i+1) \alpha_{\widehat{n_i}} + 2n_i  \alpha_{\widehat{n_i}}]\\ \nonumber
&=  - \alpha_{\widehat{n_i}} +  \alpha_{\widehat{n_i}}=0.\nonumber
\end{align}
The desired result follows since $\mathcal{\widehat{L}}$ is symmetric.
\end{proof}
In view of Eqns.~\eqref{eqn:alpha} and~\eqref{eqn:beta}, it is easy to notice that $\alpha -2\beta \neq 0$ is equivalent to $\cof \Delta \neq 0$. Under the assumption $\cof \Delta \neq 0$, we define the constant $\lambda$, an $n$-dimensional column vector $\eta$ and a Laplacian-like matrix $\mathcal{L}$ for the graph $\mathcal{S}(n_1,n_2,\cdots,n_b)$ as follows: 
\begin{equation}\label{eqn:lambda_sq}
\ds \lambda =\frac{\beta} {\alpha -2\beta},
\end{equation} 

\begin{equation}\label{eqn:eta_sq}
\eta = \dfrac{1}{\alpha -2\beta} \left[ \begin{array}{c|c|c|c|c}
			\ds \alpha -3\beta & \alpha_{\widehat{n_1}}\mathds{1}_{n_1}^t & \alpha_{\widehat{n_2}}\mathds{1}_{n_2}^t & \cdots & \alpha_{\widehat{n_b}}\mathds{1}_{n_b}^t
		\end{array}\right]^t,
\end{equation}
and 
\begin{equation}\label{eqn:Lap}
\mathcal{L}=\dfrac{1}{\alpha-2\beta}\mathcal{\widehat{L}},
\end{equation}
where $\mathcal{\widehat{L}}$ is the matrix as defined in Eqn.~\eqref{eqn:L-hat}. By Lemma~\ref{lem:L-hat1=0}, $\mathcal{L}$ is a  Laplacian-like matrix  as  $\mathcal{L}\mathds{1} = \mathbf{0}$ and $\mathds{1}^t \mathcal{L} = \mathbf{0}$. We now prove a few lemmas involving $\lambda, \eta$ and $\mathcal{L}$ that help us to compute the inverse of $\Delta$ in the desired form. 
 \begin{lem}\label{lem:sq-nu-1}
Let $\mathcal{S}(n_1,n_2,\cdots,n_b)$ be a starlike block graph  and $\mathcal{S}(n_1,n_2,\cdots,n_b) \neq \mathcal{S}(1,1)$.  If $\Delta$ is the squared distance matrix of $\mathcal{S}(n_1,n_2,\cdots,n_b)$, then $\Delta\eta = \lambda \mathds{1}$, where $\lambda$ and $\eta$ is as defined in Eqns.~\eqref{eqn:lambda_sq} and~\eqref{eqn:eta_sq}, respectively.
\end{lem}
\begin{proof}
Let $\nu = \Delta\eta $. For $v\in V_0$, we get $\nu(v)=\dfrac{1}{\alpha-2\beta}  \sum_{k=1}^b n_k  \alpha_{\widehat{n_k}} = \dfrac{\beta}{\alpha-2\beta}=\lambda.$ Next, for $1\leq i \leq b$ and  $v\in V_i$, we have
\begin{align*}
\nu(v)&= \dfrac{1}{\alpha-2\beta}\left[ (\alpha -3\beta) + (n_i-1)\alpha_{\widehat{n_i}}+ 4 \sum_{k=1\atop k\neq i}^b n_k  \alpha_{\widehat{n_k}}\right]\\
&=\dfrac{1}{\alpha-2\beta}\left[ (\alpha -(n_i-1))\alpha_{\widehat{n_i}} -3\beta +  4\sum_{k=1\atop k\neq i}^b n_k  \alpha_{\widehat{n_k}}\right]\\
&=\dfrac{1}{\alpha-2\beta}\left[ 4n_i \alpha_{\widehat{n_i}} -3\beta + 4\sum_{k=1\atop k\neq i}^b n_k  \alpha_{\widehat{n_k}}\right]\\
&=\dfrac{1}{\alpha-2\beta}\left[ 4\beta -3\beta \right]\\
&= \dfrac{\beta}{\alpha-2\beta}=\lambda.
\end{align*}
This concludes the proof.
\end{proof}

\begin{lem}\label{lem:Lsq+I=nu}
Let $\mathcal{S}(n_1,n_2,\cdots,n_b)$ be a starlike block graph  and $\mathcal{S}(n_1,n_2,\cdots,n_b) \neq \mathcal{S}(1,1)$.  If $\Delta$ is the squared distance matrix of $\mathcal{S}(n_1,n_2,\cdots,n_b)$ and  $\mathcal{L}$ is the Laplacian-like matrix  defined in Eqn.~\eqref{eqn:Lap},  then $\mathcal{L}\Delta + I= \eta \mathds{1}^t$, where $\eta$ is as defined in Eqn.~\eqref{eqn:eta_sq}.
\end{lem}
\begin{proof}
The definitions of $\Delta$, $\mathcal{L}$ and $\eta$ are based on the vertex partition $\pi=\{V_0,V_1,V_2,\cdots, V_b\}$. Therefore, we now consider various cases depending on the partition $\pi$ to prove this result.\\

\noindent \underline{\textbf{Case 1.}} For $u=v$\\

For this case, $(\mathcal{L}\Delta + I)_{uu}= 1+ (\mathcal{L}\Delta)_{uu}= 1 + \dfrac{ 1}{\alpha-2\beta}(\mathcal{\widehat{L}}\Delta)_{uu}$. We now consider following sub cases:

 \underline{\textbf{Subcase 1.1}} For $u=v$ and $u \in V_0.$
\begin{align*}
(\mathcal{\widehat{L}}\Delta)_{uu} &= - \sum_{k=1}^b n_k  \alpha_{\widehat{n_k}}=-\beta.
\end{align*}
Then, $(\mathcal{L}\Delta + I)_{uu} = 1-\dfrac{\beta}{\alpha-2\beta}=\dfrac{\alpha-3\beta}{\alpha-2\beta}= \eta(u).$\\

 \underline{\textbf{Subcase 1.2}} For $u=v$ and $v \in V_i$ for $1\leq i\leq b.$

\begin{align*}
(\mathcal{\widehat{L}}\Delta)_{uu} &= -  \alpha_{\widehat{n_i}} + (n_i-1)[6\beta_{\widehat{n_i}}- \alpha_{\widehat{n_i}}] + 8 \sum_{k=1\atop k\neq i}^b n_k  \alpha_{\widehat{n_i n_k}}\\
& = -  \alpha_{\widehat{n_i}} + (n_i-1)[6\beta_{\widehat{n_i}}- \alpha_{\widehat{n_i}}] + 8 \beta_{\widehat{n_i}}\\
&= [-1-(n_i-1)] \alpha_{\widehat{n_i}} + [6(n_i-1)+8]\beta_{\widehat{n_i}}\\
&= -n_i\alpha_{\widehat{n_i}} + 2(3n_i+1)\beta_{\widehat{n_i}}\\
&= -n_i\alpha_{\widehat{n_i}} + 2 [\beta -n_i\alpha_{\widehat{n_i}}]\\
&= [-(3n_i +1)\alpha_{\widehat{n_i}} + 2\beta ]+\alpha_{\widehat{n_i}}\\
&= -(\alpha-2\beta)+ \alpha_{\widehat{n_i}}.
\end{align*}
Then, $(\mathcal{L}\Delta + I)_{uu} = 1+ \dfrac{-(\alpha-2\beta)+ \alpha_{\widehat{n_i}}}{(\alpha-2\beta)}= \dfrac{ \alpha_{\widehat{n_i}}}{\alpha-2\beta}=\eta(u).$\\

\noindent \underline{\textbf{Case 2.}} For $u\neq v.$\\

For this case,  $(\mathcal{L}\Delta + I)_{uv}= (\mathcal{L}\Delta)_{uv}= \dfrac{1}{\alpha-2\beta}(\mathcal{\widehat{L}}\Delta)_{uv}$. We now consider following sub cases: \\

\underline{\textbf{Subcase 2.1}} For $u\neq v$ and $u\in V_0, v\in V_i$ for $1\leq i\leq b.$
\begin{align*}
(\mathcal{\widehat{L}}\Delta)_{uv} &= \beta -(n_i-1) \alpha_{\widehat{n_i}} -4 \sum_{k=1 \atop k\neq i}^b n_k  \alpha_{\widehat{ n_k}}\\
&=  \beta -(n_i-1) \alpha_{\widehat{n_i}} -4 \sum_{k=1 }^b n_k  \alpha_{\widehat{ n_k}} + 4n_i \alpha_{\widehat{n_i}}\\
&= \beta + (4n_i-(n_i-1))\alpha_{\widehat{n_i}} -4\beta\\
&= (3n_i+1)\alpha_{\widehat{n_i}} -3\beta\\
&=\alpha -3\beta.
\end{align*}
Then, $(\mathcal{L}\Delta + I)_{uv}= \dfrac{\alpha - 3\beta}{\alpha - 2\beta}=\eta(u)$.\\

\underline{\textbf{Subcase 2.2}} For $u\neq v$ and $u\in V_i, v\in V_0$ for $1\leq i\leq b.$\\
$$(\mathcal{\widehat{L}}\Delta)_{uv} =  n_i (6\beta_{\widehat{n_i}} -\alpha_{\widehat{n_i}}) +(\alpha-2\beta) + 2 \sum_{k=1\atop k\neq i}^{b}n_k \alpha_{\widehat{n_in_k}}.$$
From the calculation of Eqn.~\eqref{eqn:L.1=0}, we get
$$(\mathcal{\widehat{L}}\Delta)_{uv} =  n_i (6\beta_{\widehat{n_i}} -\alpha_{\widehat{n_i}}) +(\alpha-2\beta) + 2 \sum_{k=1\atop k\neq i}^{b}n_k \alpha_{\widehat{n_in_k}}=\alpha_{\widehat{n_i}}$$ and hence $(\mathcal{L}\Delta + I)_{uv}= \dfrac{ 1}{\alpha-2\beta}(\mathcal{\widehat{L}}\Delta)_{uv} = \dfrac{ \alpha_{\widehat{n_i}}}{(\alpha-2\beta)}=\eta(u).$\\

\underline{\textbf{Subcase 2.3}} For $u\neq v$ and $u,v\in V_i$ for $1\leq i\leq b.$
\begin{align*}
(\mathcal{\widehat{L}}\Delta)_{uv} &= -  \alpha_{\widehat{n_i}} +(\alpha-2\beta) + (n_i-1)[6\beta_{\widehat{n_i}}- \alpha_{\widehat{n_i}}] + 8 \sum_{k=1\atop k\neq i}^b n_k  \alpha_{\widehat{n_i n_k}}\\
&=\left[ -  \alpha_{\widehat{n_i}} + (n_i-1)[6\beta_{\widehat{n_i}}- \alpha_{\widehat{n_i}}] + 8 \sum_{k=1\atop k\neq i}^b n_k  \alpha_{\widehat{n_i n_k}}\right] + (\alpha-2\beta).
\end{align*}
From the calculations of Subcase~$1.2$, we get
$$(\mathcal{\widehat{L}}\Delta)_{uv} =[  - (\alpha-2\beta) + \alpha_{\widehat{n_i}}]+(\alpha-2\beta)=\alpha_{\widehat{n_i}}.$$
Then, $(\mathcal{L}\Delta + I)_{uv}= \dfrac{ 1}{\alpha-2\beta}(\mathcal{\widehat{L}}\Delta)_{uv}= \dfrac{ \alpha_{\widehat{n_i}}}{\alpha-2\beta}=\eta(u).$\\

\underline{\textbf{Subcase 2.4}} For $u\neq v$ and $u,\in V_i, v\in V_j$ for $1\leq i\leq b.$
\begin{align*}
(\mathcal{\widehat{L}}\Delta)_{uv} &= -  \alpha_{\widehat{n_i}} +4(\alpha-2\beta)+ 4 n_i(6\beta_{\widehat{n_i}}-  \alpha_{\widehat{n_i}})+2(n_j-1) \alpha_{\widehat{n_in_j}} +8 \sum_{k=1 \atop k\neq i,j}^b n_k \alpha_{\widehat{n_in_k}}\\
&=-  \alpha_{\widehat{n_i}} +4(\alpha-2\beta)+ 4 n_i(6\beta_{\widehat{n_i}}-  \alpha_{\widehat{n_i}})+2(n_j-1) \alpha_{\widehat{n_in_j}} +8 [\beta_{\widehat{n_i}} -n_j  \alpha_{\widehat{n_in_j}}]\\
&= -(4n_i+1)\alpha_{\widehat{n_i}}  + 4(\alpha-2\beta) + 8 (3n_i+1)\beta_{\widehat{n_i}} + [2(n_j-1)-8n_j] \alpha_{\widehat{n_in_j}}\\
&=-[\alpha+n_i\alpha_{\widehat{n_i}}] + 4(\alpha-2\beta) + 8 (\beta- n_i\alpha_{\widehat{n_i}}) -2(3n_j+1) \alpha_{\widehat{n_in_j}}\\
&= 3 \alpha - 9 n_i\alpha_{\widehat{n_i}}- 2 \alpha_{\widehat{n_i}}\\
&= 3\alpha - 3(3n_i+1)\alpha_{\widehat{n_i}} +  \alpha_{\widehat{n_i}}\\
&= 3\alpha - 3\alpha +  \alpha_{\widehat{n_i}}= \alpha_{\widehat{n_i}}.
\end{align*}
Then, $(\mathcal{L}\Delta + I)_{uv}=\dfrac{ \alpha_{\widehat{n_i}}}{\alpha-2\beta}=\eta(u).$ 

Combining the conclusion from all the above cases, the desired result follows.
\end{proof}

We are now ready to prove the inverse of the squared distance
matrix  $\mathcal{S}(n_1,n_2,\cdots,n_b)$ as a rank-one perturbation of a Laplacian-like matrix $\mathcal{L}$ as defined in Eqn.~\eqref{eqn:Lap}.
\begin{theorem}\label{thm:Delta_inv}
Let $\mathcal{S}(n_1,n_2,\cdots,n_b)$ be a starlike block graph and $\mathcal{S}(n_1,n_2,\cdots,n_b) \neq \mathcal{S}(1,1)$. If $\Delta$ is the squared distance matrix of the starlike block graph $\mathcal{S}(n_1,n_2,\cdots,n_b)$, then  $\Delta$ is an invertible matrix and
$$\Delta^{-1} = - \mathcal{L} + \frac{1}{\lambda} \eta \eta^t,$$ 
where $\mathcal{L}$ is the Laplacian-like matrix  defined in Eqn.~\eqref{eqn:Lap}, $\lambda$ and $\eta$ is as defined in Eqns.~\eqref{eqn:lambda_sq} and~\eqref{eqn:eta_sq}, respectively.
\end{theorem}

\begin{proof} 
Using Lemma~\ref{lem:sq-nu-1} yields $\eta^t \Delta=\lambda \mathds{1}^t$, which implies that $\eta \eta^t \Delta= \lambda \eta \mathds{1}^t$. From Corollary~\ref{cor:det},  $\Delta$ is an invertible matrix. Further, Eqns.~\eqref{eqn:beta} and~\eqref{eqn:lambda_sq} gives $\det \Delta \neq 0$ is equivalent to $\lambda\neq 0$, and hence  by Lemma~\ref{lem:Lsq+I=nu} we have $ \mathcal{L}\Delta+I=\eta \mathds{1}^t =\dfrac{1}{\lambda} \eta \eta^t \Delta$.  Therefore, $\Delta^{-1} = - \mathcal{L} + \dfrac{1}{\lambda} \eta\eta^t$.
\end{proof}

We conclude this section with a result that determines a few properties of the Laplacian-like matrix $\mathcal{L}$ as defined in Eqn.~\eqref{eqn:Lap}.

\begin{theorem}\label{thm:propery-L}
Let $\mathcal{S}(n_1,n_2,\cdots,n_b)$ be a starlike block graph  and $\mathcal{S}(n_1,n_2,\cdots,n_b) \neq \mathcal{S}(1,1)$. If  $\mathcal{L}$ is the Laplacian-like matrix  defined in Eqn.~\eqref{eqn:Lap} and $\mathcal{S}(n_1,n_2,\cdots,n_b)$ is on $n =1+\sum_{i=1}^b  n_i  $ vertices then $\mathcal{L}$ is a positive semidefinte matrix with $\textup{rank} (\mathcal{L})= n-1$. Furthermore, the cofactors of any two elements of $\mathcal{L}$ are equal to $\dfrac{1}{\alpha-2\beta}$, where $\alpha$ and $\beta$ as defined in Eqns.~\eqref{eqn:alpha} and~\eqref{eqn:beta}.
\end{theorem}
\begin{proof}
Let $\Delta$ be the squared distance matrix of $\mathcal{S}(n_1,n_2,\cdots,n_b)$. Then, by Theorem~\ref{thm:Delta_inv}, we have $\Delta^{-1}= - \mathcal{L} + \frac{1}{\lambda} \eta \eta^t$. If the eigenvalues of $\Delta^{-1}$ and $- \mathcal{L}$ are arranged in decreasing order as in Eqn.~\eqref{eqn:ev-hermitian}, then using Theorem~\ref{thm:interlacing}, we get
$$\lambda_2(- \mathcal{L} )\leq \lambda_2(\Delta^{-1}) \leq \lambda_1(- \mathcal{L} )\leq \lambda_1(\Delta^{-1}).$$
Since $\mathcal{L}\mathds{1} = \mathds{1}^t \mathcal{L}= \mathbf{0}$, it follows that $0$ is an eigenvalue of $\mathcal{L}$. By Theorem~\ref{thm:in_sq}, we have $\textup{In}(\Delta)=(1,0,n-1)$ and hence all the eigenvalues of  $\mathcal{L}$ are  non-negative and 0 is a simple eigenvalue. Therefore,  $\mathcal{L}$ is a positive semidefinte matrix with $\textup{rank} (\mathcal{L})= n-1$.

Next,  using the property $\det(A + \mathbf{u}\mathbf{v}^t) = \det(A) + \mathbf{v}^{t}\textup{ Adj }A \ \mathbf{u}$, we have 
$$\det(\Delta^{-1}) = \det(- \mathcal{L}) + \frac{1}{\lambda} \eta^t \textup{ Adj }(- \mathcal{L}) \eta.$$
Using  $\textup{rank} (\mathcal{L})= n-1$, we have $\ds \det(\Delta^{-1}) =  \frac{1}{\lambda} \eta^t \textup{ Adj }(- \mathcal{L}) \eta = \frac{(-1)^{n-1}}{\lambda} \eta^t \textup{ Adj }( \mathcal{L}) \eta.$ Since $\mathcal{L}$ is a symmetric matrix and $\mathcal{L} \mathds{1} = \mathbf{0}$, using~\cite[Lemma~$4.2$]{Bapat}  the cofactors of any two elements of $\mathcal{L}$ are equal, say $c$. Then, 
\begin{equation}\label{eqn:cof-of-L}
\det(\Delta^{-1}) = \frac{(-1)^{n-1}}{\lambda} \eta^t (cJ) \eta = \frac{(-1)^{n-1} c}{\lambda} \eta^t J \eta= \frac{(-1)^{n-1} c}{\lambda} \left(\eta^t \mathds{1}\right)^2.
\end{equation}
Using Eqns.~\eqref{eqn:beta} and~\eqref{eqn:eta_sq}, we have $\ds \eta^t \mathds{1}= \dfrac{1}{\alpha-2\beta}\left[\alpha-3\beta +\sum_{k=1}^b n_k   \alpha_{\widehat{n_k}}\right]= \dfrac{\alpha-3\beta +\beta}{\alpha-2\beta}=1.$ Substituting $\eta^t \mathds{1}=1$ in  Eqn.~\eqref{eqn:cof-of-L}, further using Corollary~\ref{cor:det} and Eqn.~\eqref{eqn:lambda_sq}, we get $$c= \dfrac{(-1)^{n-1}\lambda}{\det \Delta}=\dfrac{1}{\alpha-2\beta}.$$
This concludes the proof.
\end{proof}

In the next section, we will study the spectral radius of the squared distance matrix of the starlike $\mathcal{S}(n_1,n_2,\cdots,n_b)$.

\section{Maximization and Minimization of Spectral Radius}\label{sec:spectral-radius}

In this section, we consider the class of starlike block graphs with \( n \) vertices and \( b \) blocks. We identify the graphs in this class that uniquely maximize and minimize the spectral radius of the squared distance matrix.

Let $\Delta(G)$ be the squared distance matrix of a graph $G$. Throughout this section, to emphasise the graph $G$, we will denote the characteristic polynomial of $\Delta(G)$ as $P_{\Delta}(G,x)$, and the spectral radius of $\Delta(G)$ as $\rho(G)$. We will first recall a result that is an application of the Intermediate Value Theorem, and prove a few lemmas that play an important role in achieving our goal.

\begin{lem}\label{lem:ch-max-root}\cite[Lemma~$6.2$]{JD3}
Let $G$ and $H$ be connected graphs. Let $P_{\Delta}(G,x)$ and $P_{\Delta}(H,x)$ be the characteristic polynomials of $\Delta(G)$ and $\Delta(H)$, respectively. Then the following results hold.
\begin{itemize}
\item [$(i)$] If $x\geq \rho(G),$ then $P_{\Delta}(G,x)>0.$

\item [$(ii)$] If $P_{\Delta}(H,x)> P_{\Delta}(G,x)$ for $x\geq \rho(G),$ then $\rho(H)< \rho(G).$
\end{itemize}
\end{lem}

\begin{lem}\label{lem:f>g}
Let  $n_1,n_2,\cdots, n_b$ be positive integers such that $n_p-n_q\geq 2$ for some $1\leq p,q\leq b.$ For $x>0$, let
\begin{align*}
	f(x) &= (x+3n_p+1)(x+3n_q+1)\\
	&\qquad \times \left[x - (4x+1)\left(\frac{n_p}{x+3n_p+1} + \frac{n_q}{x+3n_q+1} + \sum_{i =1\atop i\neq p,q}^b \frac{n_i}{x+3n_i+1} \right) \right] \mbox{ and }\\
	g(x) &=  (x+3(n_p-1)+1)(x+3(n_q+1)+1)\\
	&\qquad \times \left[x - (4x+1)\left(\frac{n_p-1}{x+3(n_p-1)+1} + \frac{n_q+1}{x+3(n_q+1)+1} + \sum_{i =1\atop i\neq p,q}^b \frac{n_i}{x+3n_i+1} \right) \right].
\end{align*}
Then, $f(x)>g(x).$
\end{lem}
\begin{proof}
Considering the difference $f(x)-g(x)$, we have
\begin{align*}
f(x)-g(x) &= x\left[(x+3n_p+1)(x+3n_q+1)- (x+3(n_p-1)+1)(x+3(n_q+1)+1)\right]\\
         & \quad  \qquad  -(4x+1)\Big[n_p(x+3n_q+1)+n_q(x+3n_p+1)\\ & \qquad   \qquad   \qquad  \qquad  \qquad \quad  - (n_p-1)(x+3(n_q+1)+1) -(n_q+1)(x+3(n_p-1)+1) \Big]\\
         &\quad  \qquad  -(4x+1)\Big[ (x+3n_p+1)(x+3n_q+1)\\
         &\qquad  \qquad \qquad  \qquad \qquad  \qquad  - (x+3(n_p-1)+1)(x+3(n_q+1)+1)\Big]\sum_{i =1\atop i\neq p,q}^b \frac{n_i}{x+3n_i+1} \\
         & = [- 9(n_p-n_q-1)]x -(4x+1)[-6 (n_p-n_q-1)]\\
         & \qquad  \qquad \qquad  \qquad \qquad  \qquad  -(4x+1)[- 9(n_p-n_q-1)] \sum_{i =1\atop i\neq p,q}^b \frac{n_i}{x+3n_i+1}\\
         &= (n_p-n_q-1)[6(4x+1)-9x] + 9(4x+1)\sum_{i =1\atop i\neq p,q}^b \frac{n_i}{x+3n_i+1}.
\end{align*}
Using $n_p-n_q\geq 2$ and $x>0$, we get  $f(x)>g(x).$ 
\end{proof}

\begin{lem}\label{lem:comp_spectral}
Let  $n_1,n_2,\cdots, n_b$ be positive integers such that $n_p-n_q\geq 2$ for some $1\leq p,q\leq b.$  Then,
$$\rho(\mathcal{S}(n_1,n_2,\cdots,n_p,\cdots,n_q, \cdots,n_b)< \rho( \mathcal{S}(n_1,n_2,\cdots,n_p-1,\cdots,n_q+1, \cdots,n_b)).$$
\end{lem}
\begin{proof}
Let $H= \mathcal{S}(n_1,n_2,\cdots,n_p,\cdots,n_q, \cdots,n_b)$ and $G= \mathcal{S}(n_1,n_2,\cdots,n_p-1,\cdots,n_q+1, \cdots,n_b)$ are starlike block graphs on $n=1+\sum_{i=1}^b n_i$.  By Theorem~\ref{thm:char-poly}, for $x>0$ we can write the characteristic polynomials of $\Delta(H)$ and $\Delta(G)$ as:
$$P_{\Delta}(H,x) = \left[(x+1)^{\sum_{i=1}^b n_i-b} \prod_{i =1\atop i\neq p,q}^b (x+3n_i+1)\right] f(x), \mbox{ and }$$
and
$$P_{\Delta}(G,x) = \left[(x+1)^{\sum_{i=1}^b n_i-b} \prod_{i =1\atop i\neq p,q}^b (x+3n_i+1)\right] g(x),$$
where $f(x)$ and $g(x)$ as defined in Lemma~\ref{lem:f>g}. Thus, in view of Lemma~\ref{lem:f>g}, we have $P_{\Delta}(H,x)> P_{\Delta}(G,x)$ for all $x>0,$  and hence using Lemma~\ref{lem:ch-max-root}, $\rho(H)<\rho(G).$ 
\end{proof}

Recall, $\mathcal{S}(n_1,n_2,\cdots,n_b)$ is the starlike block graph on $n$ with a central cut vertex vertices with blocks $K_{n_1+1},K_{n_2+1},\cdots,K_{n_b+1}$ and $n=1+\sum_{i=1}^b n_i$.
For fixed value $n$ and $b$, we denote two special starlike block graphs on $n$ vertices with $b$ blocks as follows:
\begin{enumerate}
\item $\mathcal{S}_{n,b}^1=\mathcal{S}(n-b,\scriptsize{\underbrace{1,1,\ldots,1}_{b-1\ times}})$, which consists of a block with $n-b+1$ vertices, while the remaining blocks consist of 2 vertices each.

\item $\mathcal{S}_{n,b}^2=\mathcal{S}(\ceil*{\frac{n-1}{b}}, \cdots, \ceil*{\frac{n-1}{b}},\floor*{\frac{n-1}{b}},\cdots, \floor*{\frac{n-1}{b}})$, which consist of $r$ blocks with $\floor*{\frac{n-1}{b}}+1$, while the remaining blocks consist of $\ceil*{\frac{n-1}{b}}+1$ vertices each, where $n-1=\floor*{\frac{n-1}{b}}b +r$.
\end{enumerate}

We now conclude with the main result of the section. In view of Lemma~\ref{lem:comp_spectral}, the arguments of the proof of this result are similar to the \cite[Theorem~$6.5$]{JD3} and hence omitted.
\begin{theorem}
If $\mathcal{S}(n_1,n_2,\cdots,n_b)$ is a starlike block graph on $n =1+\sum_{i=1}^b  n_i =$ vertices, then 
 $$ \rho(\mathcal{S}_{n,b}^1) \leq  \rho(\mathcal{S}(n_1,n_2,\cdots,n_b)) \leq \rho(\mathcal{S}_{n,b}^2). $$ Moreover,  the maximal graph $\mathcal{S}_{n,b}^2$ and the minimal graph $\mathcal{S}_{n,b}^1$ are unique upto isomorphism.
\end{theorem}

\section{Conclusion}
In this article, we investigate the squared distance matrix \(\Delta\) of a simple connected graph with \(n\) vertices that consists of blocks \(K_{n_1 + 1}, K_{n_2 + 1}, \ldots, K_{n_b + 1}\)  with a central cut vertex. This type of graph is referred to as a starlike block graph. We identify a positive semidefinite matrix, denoted as \(\mathcal{L}\), which is Laplacian-like and has a rank of \(n-1\). Additionally, we express the inverse of \(\Delta\) as a rank-one perturbation of \(\mathcal{L}\). In addition, for fixed values of \(n\) and \(b\), we determine the extremal graphs for which the spectral radius of the squared distance matrix achieves its maximum and minimum values uniquely for starlike graphs with \(n\) vertices and \(b\) blocks. We believe that a new technique is necessary for graphs that are not trees and with multiple cut vertices. Block graphs present a suitable starting point for this exploration, and this article aims to contribute to that effort.

\section{Declarations}

\noindent{\textbf{\large Conflicts of interest}}: There is no conflict of interest.\\

\noindent{\textbf{\large  Data Availability}}:  There is no data available for this manuscript.

\small{
}


\begin{thebibliography}{20}
\bibitem{Bp1} R.B. Bapat,   A. K. Lal and S. Pati. The distance matrix of a bidirected tree, 
Electron. J. Linear Algebra. 18 (2009), 233–245.

\bibitem{Bapat}  R.B. Bapat,  Graphs and matrices, Second Edition, Hindustan Book Agency, 
New Delhi, (2014).

\bibitem{Bp3}  R.B. Bapat and  S. Sivasubramanian,
Inverse of the distance matrix of a block graph, Linear  Multilinear Algebra,  59(12) (2011), 1393-1397. 

\bibitem{Bp4} R.B. Bapat and S. Sivasubramanian, Product distance matrix of a graph and squared distance matrix of a tree, Appl. Anal. Discrete Math., 7(2) (2013), 285–301.

\bibitem{Bp5} R.B. Bapat and  S. Sivasubramanian,
Squared distance matrix of a tree: Inverse and inertia, Linear Algebra Appl.,  491 (2016),  328-342.

\bibitem{Bp6} R.B. Bapat, Squared distance matrix of a weighted tree, Electron. J. Graph Theory Appl., 7(2) (2019), 301–313.


\bibitem{JD1}J. Das and  S. Mohanty, Distance matrix of a multi-block graph: determinant and inverse, Linear  Multilinear Algebra,  70(19) (2022), 3994-4022. 


\bibitem{JD2}J. Das and  S. Mohanty.  Distance matrix of weighted cactoid-type digraphs,
Linear Multilinear Algebra, 70(20) (2022), 5392-5422.

\bibitem{JD3}J. Das and  S. Mohanty, On squared distance matrix of complete multipartite graph,
Indian J. Pure Appl. Math. 55(2) (2024), 517-537.
    
    
\bibitem{JD4}J. Das and  S. Mohanty, Inverse of the squared distance matrix of a complete multipartite graph, Electron. J. Linear Algebra 40 (2024), 475-490.

\bibitem{Gr1}  R.L. Graham and  H.O. Pollak,    
On the addressing problem for loop switching,  Bell System Tech. J., 50 (1971) , 2495-2519.



\bibitem{Gr2} R.L. Graham and  L. Lov\'asz,  
Distance matrix polynomials of trees, Adv. in Math.,  29(1) (1978), 60-88.

\bibitem{Horn} R.A. Horn and  C.R. Johnson, Matrix analysis  (Corrected reprint of the 1985 original). Cambridge University Press, Cambridge, (1990). 

\bibitem{Hou1} Y. Hou and J. Chen. 
Inverse of the distance matrix of a cactoid digraph. Linear Algebra Appl., 475 (2015), 1-10.

\bibitem{Hou2} Y. Hou, A. Fang and Y. Sun,
Inverse of the distance matrix of a cycle-clique graph, Linear Algebra Appl., 485 (2015),  33-46.

\bibitem{How} C. Howell, M. Kemton, K. Sandall and J. Sinkovic, Unicyclic graphs and the inertia of the squared distance matrix, Electron. J. Linear Algebra, 39 (2023), 491–515.  


\bibitem{IM} I. Mahato and M. R. Kannan, Squared distance matrices of trees with matrix weights, AKCE Int. J. Graphs Comb., 20 (2) (2023), 168–176.

\bibitem{Ob} M.R. Oboudi, Majorization and the spectral radius of starlike trees, J. Comb. Optim. 36 (2018),  121–129.

\bibitem{Ob1} M.R. Oboudi. Distance spectral radius of complete multipartite graphs and majorization, Linear Algebra Appl.,  583 (2019), 134-145.  

\bibitem{Sun}S. Sun and  K.C. Das,  Proof of a conjecture on distance energy change of complete multipartite graph due to edge deletion. Linear Algebra Appl. 611 (2021), 253-259.
 
\bibitem{So} W. So, A shorter proof of the distance energy of complete multipartite graphs, Spec. Matrices, 5 (2017), 61-63.





\bibitem{Zhou1} H. Zhou, The inverse of the distance matrix of a distance well-defined graph,
Linear Algebra Appl.,  517 (2017), 11-29. 

 
\bibitem{You} L. You, M. Yang, W. So, W. Xi, 
 On the spectrum of an equitable quotient matrix and its application. Linear Algebra Appl., 577, (2019), 21-40. 
 

\bibitem{Zhang1} F. Zhang (Editor). The Schur complement and its applications. Numer. Methods and Algorithms, 4. Springer-Verlag, New York, (2005).




\bibitem{Zhou2} H. Zhou, Q. Ding and R. Jia. Inverse of the distance matrix of a weighted cactoid digraph. Appl. Math. Comput.  362 (2019), 124552, 11 pp. 




\end{thebibliography}
\end{document}